\documentclass[12pt]{article}
\usepackage{amsmath, amssymb, amsthm, hyperref}
\usepackage{tikz}
\usetikzlibrary{arrows.meta, positioning}
\usepackage{tikz}
\usetikzlibrary{decorations.pathmorphing,calc}
\usepackage{bbold}
\usepackage{bbold} 
\usepackage{graphicx} 
\usepackage[paper=a4paper,margin=2cm]{geometry}
\usepackage{xcolor}
\usepackage{longtable}
\usepackage{comment}

\title{The Parry Order of Perron Numbers}
\author{Kevin G Hare\footnote{Research of K. G. Hare is supported, in part, by NSERC Grant 2025-03965.}  and Hachem Hichri}
\date{}

\theoremstyle{plain}
\newtheorem{theorem}{Theorem}[section]
\newtheorem{lemma}[theorem]{Lemma}
\newtheorem{proposition}[theorem]{Proposition}
\newtheorem{example}[theorem]{Example}
\newtheorem{corollary}[theorem]{Corollary}
\newtheorem{conj}[theorem]{Conjecture}
\counterwithin{figure}{section}
\counterwithin{table}{section}

\theoremstyle{remark}
\newtheorem{remark}[theorem]{Remark}
\newcommand{\Ord}{\mathrm{Ord}}
\newcommand{\OrdP}{\mathrm{Ord}_{P}}
\newcommand{\SpecP}{\mathrm{Spec}_{P}}

\begin{document}

\maketitle

\begin{abstract}
We introduce the \emph{Parry order} $\OrdP(\beta)$, defined as the largest integer $n$ for which $\beta^n$ is a Parry number. This leads to a natural partition of the set of Perron numbers as follows: 
\[
 \mathbb{P} = \left(\bigcup_{n \ge 0} \mathcal{H}_n\right)\cup \mathcal{H}_\infty,
\]
where $\mathcal{H}_n$ is the class of Perron numbers with Parry order $n$,
and $\mathcal{H}_\infty = S \cup T$ consists exactly of all Pisot and Salem numbers.

We show that a Perron number has infinitely many Parry powers 
if and only if it is Pisot or Salem. For every other Perron number, only finitely many powers can be Parry. 
We give an explicit upper bound on $\OrdP(\beta)$ in terms of algebraic properties of $\beta$. 
We provide explicit examples of non-Parry Perron numbers whose powers become Parry, demonstrating that several $\mathcal{H}_n$ are non-empty and structurally rich.
We give an infinite family of cubic non-Pisot numbers, all of which have finite Parry order, but where the family has unbounded Parry order.

These results establish a new dynamical perspective on Perron numbers, connecting $\beta$-expansion theory with classical questions surrounding Salem numbers and Lehmer-type conjectures.
\end{abstract}

\section*{Introduction}

The study of algebraic integers through their dynamical and arithmetic properties has long revealed deep connections between number theory, Diophantine approximation, and ergodic theory. Among such numbers, \emph{Perron}, \emph{Pisot}, and \emph{Salem numbers} occupy a central place. These numbers, defined as real algebraic integers greater than one, and are distinguished by the location of their other Galois conjugates. 
In particular, 
\begin{itemize}
 \item A \emph{Perron number} is a real algebraic integer $\beta > 0$ strictly larger in modulus than its other conjugates.
 \item A \emph{Pisot number} is a real algebraic integer $\theta > 1$ whose conjugates are strictly inside the open unit disc.
 \item A \emph{Salem number} is a real algebraic integer $\tau > 1$ whose conjugates are in the closed unit disc, with at least one lying on the boundary of the unit circle.
\end{itemize}

We denote, respectively, by $\mathbb{P}, \mathcal{S}, \mathcal{T}$ the sets of Perron, Pisot, and Salem numbers. 
Obviously, we have $\mathcal{S} \cup \mathcal{T} \subset \mathbb{P}$.

Finally, we recall that, for any polynomial, 
\[
P(x) = a \prod_{i=1}^d (x-\alpha_i) \in \mathbb{C}[x],
\] 
the \emph{Mahler measure} of $P$ is defined as
\[
M(P) = |a| \prod_{i=1}^d \max \{1, |\alpha_i|\}=\exp\!\left(\int_0^1 \log\left|P\!\left(e^{2\pi i t}\right)\right|\,dt\right).
\]

In particular, if $\alpha$ is an algebraic integer, we define $M(\alpha) := M(P_\alpha)$, where $P_\alpha$ is the minimal polynomial of $\alpha$. It is clear that $M(\alpha) = \alpha$ if and only if $\alpha$ is a Pisot or a Salem number.

For background on Pisot and Salem numbers, we refer to Bertin \cite{Bertin1992}.

The study of these numbers is motivated by their connections with several problems in algebraic number theory, Diophantine approximation, distribution modulo $1$, ergodic theory, and Fourier analysis (see, e.g~\cite{SA07, GH99, Sal63, JS97, Sch80}).

In this paper we study the arithmetic nature of these numbers arises from the theory of $\beta$-expansions.  These were introduced by Rényi \cite{Ren57} and studied by Parry \cite{Par60}. For a real number $\beta > 1$, the $\beta$-expansion of $1$ encodes both the algebraic and dynamical behavior of $\beta$ through the orbit of $1$ under the $\beta$-transformation. 
The $\beta$-transformation $T_\beta$ is given by $T_\beta(x) = \beta x \mod 1$.
The greedy expansion is given by the sequence $d_1 d_2 d_3 \dots$ with the property that $x = \sum_{k=1}^\infty \frac{d_k}{\beta^k}$ and $d_k = \lfloor\beta\cdot T_\beta^{[k-1]}(x)\rfloor$. 
In particular, when this orbit $\{T^{[k]}(1)\}_{k=1}^\infty$ is finite, $\beta$ is called a \emph{Parry number}. We say $\beta$ is a \emph{simple Parry number} if the orbit is eventually $0$. The set of Parry numbers is denoted by $\mathcal{P}$, and its complementary set in $\mathbb{P}$ is denoted by $\mathbb{P}_a$.

Schmidt \cite{Sch80}, proved that all Pisot numbers are Parry numbers and claimed that some computational results led him to expect that similar result should be true for all Salem numbers---a statement now referred to as \emph{Schmidt's conjecture}.
Boyd showed that all degree 4 Salem numbers are Parry numbers.
In contrast to Schmidt's conjecture, Boyd then conjectured that there exist higher degree Salem numbers that are not Parry numbers.
See \cite{Boyd1989}.

The algebraic characterization of Parry numbers remains an open problem. The first important result in this direction is due to Parry.  Parry showed that all Parry numbers are Perron numbers with Galois conjugates of modulus $\leq 2$ \cite{Par60}. Solomyak improved this result by giving a complete description of the set $\Phi$ defined as the closure of the set of all conjugates of the Parry numbers (excluding the Parry numbers themselves) \cite{Sol94}. For our purposes, it will be sufficient to recall that the $\Phi$ contains the unit disk $\mathbb{D}$ and it is contained in the disk of radius the golden ratio $\theta_0=\frac{1+\sqrt{5}}{2}$. Consequently, any Perron number having another Galois conjugate of modulus $>\theta_0$ is a non-Parry number. This will be called {\em Solomyak's criteria}.

Let $\beta$ be a simple Parry number, and $a_1 \dots a_k$ be the finite greedy $\beta$-expansion of $1$.
Let $R(x)= x^k - a_1 x^{k-1} - \dots - a_k$.
We note that $R(\beta) = 0$ and hence is divisible by the minimal polynomial of $\theta$.
Let $a_1 a_2 \dots a_k (a_{k+1} \dots a_{k+\ell})^\omega$ be a periodic greedy expansions of $1$ of a Parry number $\theta$.
We define $R(x) = \left(x^{k+\ell} - a_1 x^{k+\ell} - \dots - a_{k+\ell}\right) -
    \left(x^k - a_1 x^{k-1} - \dots - a_k\right)$.
As before, we note that $R(\beta) = 0$ and is divisible by the minimal polynomial of $\theta$.
Notice that the tail coefficient of $R(x)$ is bounded by $\lfloor \theta \rfloor$ in absolute value.
Consequently, any Perron number which is strictly less than the absolute value of the tail coefficient of it's minimal polynomial is a non-Parry number.
This will be called {\em Boyd's criteria}.
Here, $R(x)$ is often called the {\em companion polynomial}.

This allows us to partition the set of Perron numbers into the following two subsets : $\mathbb{P}=\mathbb{P}_\Phi\cup \overline{\mathbb{P}}_\Phi$, where $\mathbb{P}_\Phi$ is the set of Perron numbers with all its Galois conjugates lying in the interior of $\Phi$ and $\overline{\mathbb{P}}_\Phi$ its complementary set. Note that this partition is different from $\mathbb{P}=\mathcal{P}\cup \mathbb{P}_a$, since it is already known that $\overline{\mathbb{P}}_\Phi\subsetneq \mathbb{P}_a$ and $\mathcal{P}\subsetneq\mathbb{P}_\Phi$. A major challenge to answer many open questions about Parry and non-Parry numbers is to give a complete characterization of the elements of $\mathbb{P}_\Phi\setminus \mathcal{P}$. It should be noted that, in general, it is not easy to confirm that a given Perron from $\mathbb{P}_\Phi$ is not a Parry number. To our knowledge, the only known infinite sequence of $\mathbb{P}_\Phi\setminus \mathcal{P}$ is given by Akiyama \cite{Akiyama2016}. 
This is discussed further in Section \ref{sec:ParryVsNonParry}.

From a result of Lind~\cite[Proposition~1]{Lind1984}, we know that the set of Perron numbers is closed
under both addition and multiplication. Consequently,  the set of Perron numbers is closed under the maps
$
\beta \longmapsto \beta^k $, for all positive integers $k$. However, these maps give rise to a rather mysterious dynamics  between the subsets of 
Parry  and non-Parry  Perron numbers, which together form a
partition of the set of Perron numbers.

The objective of our study is to understand this dynamics and to shed light on
the action of these maps on these two subclasses.
In particular, we are interested in the \emph{powers of Perron numbers}, focusing on how the Parry property behaves under exponentiation. Specifically, we are concerned with the following question: when powers of a given Perron number $\beta$ can be Parry or non-Parry number?

We found that this question connects the combinatorial structure of $\beta$-expansions with the arithmetic geometry of algebraic integers, and in particular, with classical conjectures such as those of Lehmer, Boyd, and Schmidt.

In Section \ref{sec:1}, our first result establishes the following precise criterion: {for a Perron number $\beta$, infinitely many powers $\beta^m$ are Parry numbers if and only if $\beta$ is a Pisot or a Salem number}. This characterization bridges a dynamical property of Perron numbers with the algebraic nature of Pisot and Salem numbers. Moreover, for non-Pisot and non-Salem numbers, we provide an explicit bound beyond which $\beta^n$ can no longer be a Parry number.  Such a bound is governed by the Mahler measure of $\beta$.

In Section \ref{sec:2}, we explore some density and structural aspects for these subclasses of Perron numbers by introducing the notion of the \emph{Parry order of $\beta$} and \emph{Parry spectrum of $\beta$}, denoted by $\OrdP (\beta)$ and $\SpecP(\beta)$ respectively.  This determines exactly how often the powers of a given Perron number $\beta$ remain a Parry number. This leads to a natural new partition of the set of Perron numbers according to their Parry order and provides a dynamical perspective to better understand the distribution of the subsets of Parry and non-Parry numbers.

In Section \ref{sec:3}, we develop a method to precisely compute the Parry order and Parry spectrum of some Perron numbers. This allows us in particular to provide several examples of Perron numbers that are non-Parry, yet admit certain powers that are Parry numbers. As far as we know, these are the first known examples satisfying this property. 
We give an analysis of cubic non-Pisot numbers, given conditions for when 
    these would be Parry.
Using this we construct Perron numbers with arbitrarily large
    Parry order.
We also study the roots of $x^n - x^{n-1}-1$ and $x^n - x- 1$.
We consider small degree examples to better understand how often Perron numbers are Parry numbers. 

In Section \ref{sec:4} some concluding remarks are made.

In this work, we discuss several conjectures motivated by our findings, including a new version of Lehmer’s problem and some density results related to Boyd's and Schmidt's conjecture's. Together, these results contribute to a unified understanding of the interplay between algebraic and dynamical properties of $\beta$-expansion theory.

\section{A dynamical version of strong Lehmer's conjecture}
\label{sec:1}

Parry showed in \cite{Par60}, that the set of simple Parry numbers is dense in $(1,\infty)$. Moreover, the first author and Smyth showed in \cite{HichriSmyth2024} that the set of non-Parry Perron numbers is also dense in $(1,\infty)$. In the same direction and as a consequence of Parry result, we start by showing the following result:
\begin{lemma} 
\label{parrydensity}
 The set of non-simple Parry numbers is dense in $(1,\infty)$. 
\end{lemma}
\begin{proof}
 Let $r>1$ be a real number.  According to Parry, there exists an infinite sequence of simple Parry numbers, say $(\beta_n)$, converging to $r$. For a fixed integer $n \geq 1$, we denote respectively the beta expansion of $\beta_n$ and $\beta_{n+1}$ by:
 $d_{\beta_n}(1)=c_1c_2\ldots c_m$ and $d_{\beta_{n+1}}(1)=d_1d_2\ldots d_{\ell}$.

Next, we consider the first index $j$ such that $c_j \neq d_j$.
Assume that $c_j<d_j$, (the case $c_j > d_j$ is similar). 
For a sufficiently large positive integer $m$, the sequence $c_1\dots c_{j-1} d_j \underbrace{0\dots0}_{m} (c_1\dots c_{j-1} \underbrace{0\dots0}_m)^{\omega}$
gives a self-admissible sequence. So, it is the $\beta$-expansion of a real number, say $f_n$. Taking $m$ large, the corresponding
number $f_n$ is between $\beta_n$ and $\beta_{n+1}$ and it is a non-simple Parry number.
Moreover, it is clear that the sequence $(f_n)$ converges to $r$.
\end{proof}

From \cite{Par60} the set of simple Parry numbers is dense in $(1,\infty)$.
The set of Pisot numbers is a closed set \cite{Sal63} and 
a Salem number cannot be a simple Parry number \cite{Boyd1989}.
This shows that there is a dense set of Parry numbers in $(1,\infty)$ that are neither Salem nor Pisot numbers. Among the set of Perron numbers, Salem and Pisot numbers can be characterized by the following property:

\begin{proposition}\label{Mainprop}
 Let $\beta >1 $ be a Perron number. Then $\beta^n$ is a Parry number for infinitely many integers $n\geq 1$ if and only if $\beta $ is a Salem or Pisot number. 
\end{proposition}

\begin{proof}
Let $\beta >1 $ be a Perron number such that $\beta^n$ is a Parry number for infinitely many integers $n\geq 1$. Assume that $\beta $ is neither a Salem number nor a Pisot number. Then $\beta$ has at least one Galois conjugate $\gamma$ other than $\beta $ satisfying $\lvert\gamma \rvert >1$. We denote by $N$ the normal closure of $\mathbb{Q}(\beta)$. As the automorphism group of $N/\mathbb{Q}$ is transitive on the conjugates of $\beta$, there is an automorphism $\sigma$ mapping $\beta$ to $\gamma$. So $\sigma$ maps $ \beta^n$ to $\gamma^n$, and consequently $\gamma^n$ will be a Galois conjugate of $\beta^n$ other than $\beta ^n$, for all $n\geq 1$. However, as $\lvert\gamma \rvert >1$, there exists an integer $n_0$ such that for all $n\geq n_0$ we have $\lvert\gamma^n\rvert > \frac{1+\sqrt{5}}{2}$. But this implies that $\beta^n $ can be a Parry number only for finitely many integers $n$. Hence, all the other conjugates of $\beta $ should be of modulus $\leq 1$. Thus, $\beta$ is either a Salem or Pisot number.

 Conversely, it is known that if $\beta$ is a Pisot number, then $\beta^n$ is also a Pisot number and hence a Parry number for all $n\geq 1$.

 Finally, let $\beta$ be a Salem number.
 By \cite[Theorem 1]{aH22} there are infinitely many integers $n$ such that $\beta^n$ is a Parry number. 
\end{proof}

\begin{remark}
We recall that Schmidt \cite{Sch80} gave a \textit{quasi}-dynamical characterization of Pisot and Salem numbers among all real numbers larger than 1 by showing that: 
If $\beta>1$ is a real number such that every $x \in \mathbb{Q}\cap[0,1)$ has a periodic $\beta$-expansion, then $\beta$ is either a Pisot number or a Salem number. Although Proposition \ref{Mainprop} provides a complete dynamical characterization of Pisot and Salem numbers, we should note that, in contrast to Schmidt's result, Proposition \ref{Mainprop} is no longer true if replace ``Perron number'' by real number or even algebraic integer. Indeed, if $\beta $ is a Pisot number, then it is clear that there exists an integer $m$ such that $\alpha =\sqrt[m]{\beta}$ is neither a Pisot nor a Salem number, although $\alpha^n$ is a Parry number for infinitely many integers $n$.
\end{remark}

 We denote by $\rho\approx 1.32$ the smallest Pisot number. Here, $\rho$ is often called the plastic ratio.
As an immediate consequence of Proposition \ref{Mainprop}, we have what follows:
\begin{corollary}{\ }
 \label{corr power}
 \begin{enumerate}
 \item Let $\beta \in (1,\rho)$ be a simple Parry number. Then $\beta^n$ is a non-Parry number for all sufficiently large integers $n$.  \label{prt:corr power 1}
 \item For any Perron number $\beta \in (1,\rho) $ we have: $\beta^n$ is a Parry number for infinitely many integers $n$ if and only if $\beta $ is a Salem number. \label{prt:corr power 2}
 \end{enumerate}
\end{corollary}

\begin{proof}
 As a Salem number cannot be a simple Parry number \cite{Boy89} and obviously $\beta $ is not a Pisot number, then $\beta^n $ can be a Parry number for only finitely many integers $n$. 
\end{proof}

\begin{remark}
 In 1933, Lehmer asked whether the Mahler measure of any nonzero noncyclotomic irreducible polynomial with integer coefficients is bounded below
by some constant $c > 1$. This is now commonly referred to as ``Lehmer’s
problem'' or (inaccurately) as ``Lehmer’s conjecture''. Up to now, the smallest known Salem number is $\tau_{10}\approx 1.176$ with minimal polynomial $L(x)=x^{10}+x^9-x^7-x^6-x^5-x^4-x^3+x+1$ discovered by D. H. Lehmer \cite{Leh33} The \emph{strong version} of Lehmer's conjecture states
that in fact \(c=\tau_{10}\). 

There are many variants and equivalent version of Lehmer's conjecture \cite{VergerGaugry2019}. Here, using Corollary \ref{corr power} part \ref{prt:corr power 2} we give a dynamical version of the strong version of Lehmer's conjecture:
\end{remark}
\begin{conj}
 If $\beta < c = \tau_{10}$ is a Perron number, then $\beta^n$ can be a Parry number for only finitely many integers $n$. 
\end{conj}

\section{Dynamical dispersion of Perron numbers}
\label{sec:2}

\subsection{Parry order of a Perron number}
According to Proposition \ref{Mainprop}, if $\beta$ is a Perron number which is neither a Pisot nor a Salem number, then the set $\{n\in \mathbb{N} ~~\mathrm{such~that} ~~ \beta^n \in \mathcal{P}\}$ is finite and hence the set is bounded. This allows us to define the Parry spectrum and Parry order of such a Perron number $\beta$ as follows:
\begin{align*}
\SpecP(\beta) & = \{n \in \mathbb{N} \text{ such that } \beta^n \in \mathcal{P}\} \\
\OrdP(\beta)& = \max\{n\in \mathbb{N} \text{ such that } \beta^n \in \mathcal{P}\}.
\end{align*}

If for all positive integers $n$, we have $\beta^n $ is a non-Parry number, we set $\SpecP(\beta) = \emptyset$ and $\OrdP(\beta)=0$. On the other hand, a natural extension of the notion of Parry order to a Pisot or Salem number $\beta$ is to set 
$\OrdP(\beta)=\infty$ which is consistent with the definition of the Parry order and with Proposition \ref{Mainprop}.
In the case where $\beta$ is a Pisot number, $\SpecP(\beta) = \mathbb{N}$.
The structure when $\beta$ is a Salem number is potentially more complicated.

 In the following theorem we show that, if $ \OrdP(\beta)$ is finite, then it is bounded above in terms of the degree and the Mahler measure of $\beta$.

 We first recall the following lemma which can be found \cite{HichriSmyth2024}: 
\begin{lemma}
 \label{lem:perronpowers}
 Given a Perron number $\alpha$, all its powers $\alpha^k$ $(k \ge 1)$ are Perron numbers having the same degree as $\alpha$.
\end{lemma}

\begin{theorem}\label{Orbmaj}
 Let $\theta_0 = (1+\sqrt{5})/2$ be the golden ratio. If $\beta$ is a Perron number of degree $d$ which is neither a Pisot nor a Salem number, then
 $$
 \OrdP(\beta) \leq \frac{d \log (\theta_0)}{\log( M(\beta)) - \log(\beta)}:=\mathcal{N}_p(\beta)
 $$
 In particular, if $\mathcal{N}_p(\beta)<1$ then $\beta $ is a non-Parry number.
\end{theorem}
\begin{proof}
Let $\beta = \beta_1, \beta_2, \dots, \beta_s$ be the conjugates of a Perron number $\beta$ such that $|\beta_2|\geq |\beta_3| \geq \ldots \geq |\beta_s| \geq 1$, then we have:
\[
M(\beta) = \prod_{i=1}^s |\beta_i|.
\]
Since $\beta$ is neither Pisot nor Salem, $2 \leq s \leq d$ and
\[
1 < \frac{M(\beta)}{\beta} = \prod_{i=2}^s |\beta_i| \leq |\beta_2|^{s-1} \leq |\beta_2|^d.
\]
which gives $1<(M(\beta)/\beta)^{1/d} < |\beta_2|$.
Therefore, for $n$ sufficiently large, one has
\[
\theta_0 < \left( \frac{M(\beta)}{\beta} \right)^{n/d} < |\beta_2|^n.
\]
However, according to Lemma \ref{lem:perronpowers}, we know that $\beta_2^n$ is a conjugate of $\beta^n$. Thus, as $|\beta_2^n| > \theta_0$, we get from Solomyak's criteria that $\beta^n$ is not a Parry number for all such integers $n$. From the above inequalities we easily get $\OrdP(\beta) \leq \frac{d \log (\theta_0)}{\log( M(\beta)) - \log(\beta)} $
\end{proof}

One could alternately use the size of the largest conjugate of $\beta$ to bound $\Ord_p(\beta)$.

\begin{lemma}
 \label{lem:K}
 Let $\theta_0 = (1+\sqrt{5})/2$ be the golden ratio. 
 Let $\beta$ be a Perron number that is neither a Pisot nor a Salem number.
 Let $\gamma $ be a Galois conjugate of $\beta$ of modulus $>1$ other than $\beta$. Then for all $k\geq K(\beta):=\left\lceil \frac{\log{\theta_0}}{\log|\gamma|}\right\rceil$ we have $\beta^k$ is not a Parry number.
\end{lemma}

\begin{proof} Note that the existence of a conjugate of $\gamma> 1$ other than $\beta $ is due to the fact that $\beta $ is a Perron number which is neither a Pisot nor a Salem number. 
 Take $K(\beta) := \left\lceil \frac{\log{\theta_0}}{\log|\gamma|}\right\rceil$ where $\theta_0= \frac{1+\sqrt{5}}{2}$ is the golden ratio. Then , for all $k\geq K(\beta)$ we have $|\gamma|^k\geq \theta_0$. By Lemma \ref{lem:perronpowers}, $\gamma^k$ is a Galois conjugate of $\beta^k$. Hence, $\beta^k $ is not a Parry number for all $k\geq K(\beta) $ by Solomyak criteria.
\end{proof}

\begin{remark}
 In fact, it is possible to replace the value of $K(\beta)$ with the constant $\mathcal{N}_p(\beta)$ given in Lemma \ref{Orbmaj}. However, as often $K(\beta) < \mathcal{N}_p(\beta)$, using $K$ will be significantly less computationally expensive. For the same reason, it will be more practical to choose $\gamma$ satisfying
\[
 |\gamma| = \max\{|\alpha| : \alpha \text{ is an other conjugate of } \beta \text{ with } |\alpha| > 1\}.
\]
\end{remark}

In the other direction, Akiyama gave a test that can prove if $\beta$ is non-Parry.
\begin{lemma}[Akiyama \cite{Akiyama2016}] \label{lem:Akiyama}
Let $\beta$ be a Perron number which is not neither Pisot nor Salem number. Then, there exists a conjugate 
$\gamma \neq \beta$ of $\beta$ with $|\gamma| > 1$. We denote by $\sigma$ the Galois action taking $\beta$ to $\gamma$ If there exists a $k \in \mathbb{N}$ such that 
\[
\big| \sigma(T_\beta^k(1)) \big| > \frac{\lfloor \beta \rfloor}{|\gamma| - 1},
\]
then $\beta$ is not a Parry number.
 \end{lemma} 

 \begin{remark}
    Note that for any Perron number $\beta$, we always have $\OrdP(\beta) \leq K(\beta)$ (often $\OrdP(\beta) < K(\beta)$).  This provides a bound $n_0$ such that $\beta^n$ is a non-Parry number for all integers $n\geq n_0$. We show in Section \ref{sec:3}, how one can determine exactly the Parry order for several examples of Perron numbers. However, in contrast to $\mathcal{N}_p(\beta)$, it seems computatioanl infeasible to determine the value of $\OrdP(\beta)$ for every Perron number $\beta$.  
    This is because there is no bound on the length of a periodic $\beta$-expansion.   See for example \cite{Stockton21}.
 Further, if the power is not Parry, then Theorem \ref{lem:Akiyama} is not known to be if and only if.  
\end{remark}
Next, we give some properties of the Parry order of Perron numbers: 

\begin{theorem}\label{Parryorder}
Let $\beta$ be a Perron number and $n$ be a non negative integer.
\begin{enumerate}
 \item If $\beta^n \in \mathcal{P}$ then $n\leq \OrdP(\beta) $. \label{part:Parryorder 1}
 \item If $ \OrdP(\beta)=n$ and $m$ divides $n$, then $\OrdP(\beta^m)=\frac{n}{m}$. In particular, $ \OrdP(\beta^n)=1$. \label{part:Parryorder 2}
 \item If $\OrdP(\beta)=n$, then for any positive integers $d, k$, if $k>\frac{n}{d}$ then $\OrdP(\beta^k)<d$. In particular, for all integers $k>\frac{n}{2}$, we have $\OrdP(\beta^k) \in \{0,1\}$
\end{enumerate}
\end{theorem}

\begin{remark}
    It is worth noting that the converse of Theorem \ref{Parryorder} Part \ref{part:Parryorder 1} is not, in general, true.
    To see this, let $\beta$ be the Perron number defined by the polynomial $P(x)=x^{14}-x^{13}-x^{12}-x^{8}+1$ we have: $\beta^3 \notin \mathcal{P} $ although $ 3< \OrdP(\beta) =30$.
\end{remark}

\begin{proof}[Proof of Theorem \ref{Parryorder}]{\ }
 \begin{enumerate}
 \item This follows immediately form the definition of $\OrdP(\beta)$. 
 \item As $\beta ^n=(\beta ^m)^{\frac{n}{m}} \in \mathcal{P}$ then, using part \ref{part:Parryorder 1}, we have $\frac{n}{m}\leq \OrdP(\beta^m)$. 
 \item Clear. 
 \end{enumerate}
\end{proof}

\subsection[New partition of the set P]{New partition of the set \(\mathbb{P}\) }
 
 For the remainder of this paper, we define for any non-negative integer $n$ the set $\mathcal{H}_n$ of all Perron numbers of Parry order $n$.  That is:
 $$\mathcal{H}_n=\{ \beta \in \mathbb{P} ~~ \mathrm{ such ~~ that } ~~\OrdP(\beta)=n \}$$
Moreover, according to Proposition \ref{Mainprop}, we know $\OrdP(\beta)=\infty$ if and only if $\beta$ is a Pisot or a Salem number. Therefore, it is reasonable to set $\mathcal{H}_\infty=\mathcal{S}\cup \mathcal{T}.$
Obviously, these families of subsets $(\mathcal{H}_n)$ are mutually disjoint and make a partition of the set of Perron numbers $\mathbb{P}$. In particular, we have: 

\begin{equation} \label{P-decomposition}
\mathbb{P} = \left(\bigcup_{n \geq 0} \mathcal{H}_n\right) \cup \mathcal{H}_\infty.
\end{equation}

The purpose of this subsection is to provide some properties of subclasses of Perron numbers with respect to the family of subsets $(\mathcal{H}_n)$.

First we give, in the same Figure \ref{fig:placeholder}, the three partitions of the set of Perron numbers $\mathbb{P}=\overline{\mathbb{P}}_\Phi\cup\mathbb{P}_\Phi=\mathbb{P}_a\cup \mathcal{P}=\mathcal{H}_\infty \cup (\bigcup_{k\geq 0}{\mathcal{H}_k})$. All relationships between the subsets in Figure \ref{fig:placeholder} are either evident, or can be easily proved, except for the two points below.
\begin{itemize}
 \item In Figure \ref{fig:placeholder}, we assume that $\mathcal{H}_n\cap \mathbb{P}_a$ and $\mathcal{H}_n\cap \mathcal{P} $ are non-empty for all integers $n\geq 2$. In fact, we find many numerical results that support this assumption. See Section \ref{sec:3}. However, it is not clear how to prove this rigorously. For example, the Perron number $\beta$ defined by the polynomial $P(x)=x^{13}-x^{12}-x^{11}-x^{10}+1$ has 
 \begin{align*}
     \SpecP(\beta) =\  & \{[1, 2, 3, 4, 5, 7, 8, 9, 10, 11, 14, 15, 17, 19, 21, 22
, 24, 31, 47]\}
 \end{align*}
This gives that $\beta^j$ satisfying various relations.
For example,  $\beta \in \mathcal{H}_{47} \cap \mathcal{P}$,
$\beta^2 \in \mathcal{H}_{12} \cap \mathcal{P}$,
$\beta^3 \in \mathcal{H}_{8} \cap \mathcal{P}$, 
$\beta^{4} \in \mathcal{H}_{6} \cap \mathbb{P}$,
$\beta^{5} \in \mathcal{H}_{3} \cap \mathbb{P}$, and 
$\beta^{6} \in \mathcal{H}_{4} \cap \mathbb{P}_a$.
\item According to Schmidt \cite{Sch80}, all Salem numbers should be Parry numbers and then $\mathcal{H}_\infty \subset\mathcal{P}$. The opposite is conjectured by Boyd \cite{Boyd1996}. When creating Figure \ref{fig:placeholder}, we adopt Boyd's model which predicts the existence of Salem numbers which are non-Parry numbers.
\end{itemize}

\begin{remark}
 Note that, according to Theorem \ref{Parryorder}, part \ref{part:Parryorder 2}, from any Perron number $\beta $ such that $\OrdP(\beta)= n $, then $\beta $ can be used to obtain at least $\tau(n)= \# \{ d \in \mathbb{N} \text{ such that } d \mid n \}$ Perron numbers belonging to one of the subsets $\mathcal{H}_j$ with $j \geq 1 $. However, the number of such Perron numbers is often strictly larger than $\tau(n)$. For example, considering again $x^{13}-x^{12}-x^{11}-x^{10}+1$ we have that various powers of $\beta$ belong to one of 8 possible $\mathcal{H}_j$. 
 We note that $\tau(\OrdP(\beta)) = \tau(47) = 2$. 
 \end{remark}

In the following theorem, we provide a useful relationship between the subsets of Figure \ref{fig:placeholder}: 
\begin{theorem}\label{H01}{\ }
We have $\mathcal{H}_1\subsetneq \mathcal{P} \subsetneq \mathbb{P}_\Phi$.
\end{theorem}

\begin{figure} \label{figure 1}
 \centering
 \includegraphics[width=0.5\linewidth]{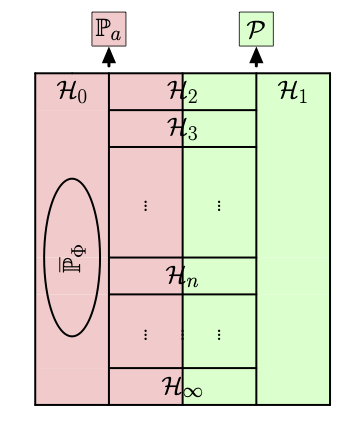}
 \caption{Partition of $\mathbb{P}=\overline{\mathbb{P}}_\Phi\cup\mathbb{P}_\Phi=\mathbb{P}_a\cup \mathcal{P}=\mathcal{H}_\infty \cup (\bigcup_{k\geq 0}{\mathcal{H}_k})$}
 \label{fig:placeholder}
\end{figure}

\begin{proof}
 By the definition of $\mathcal{H}_1$ and Solomyak's work \cite{Sol94}, it is clear that $\mathcal{H}_1\subset \mathcal{P} \subsetneq \mathbb{P}_\Phi$. To show that the first inclusion is strict, consider $x^{13}-x^{12}-x^{11}-x^{10}+1$ from before. We note that $\beta \in \mathcal{P}$ and $\beta \in \mathcal{H}_{47}$ and, therefore, $\beta \not\in \mathcal{H}_1$.
\end{proof}

\begin{remark}
  Among the families  $(\mathcal{H}_n)_{n\in \mathbb{N}}$, the only subsets closed under positive powers are $\mathcal{H}_0$ and $\mathcal{H}_\infty$. Indeed, according to the third point of Theorem \ref{Parryorder} ($d=1$), for any integers $k>n>0$, we have $\mathcal{H}_n^k \subset \mathcal{H}_0 $. 
\end{remark}

\begin{theorem}{\ }
 \begin{enumerate}
 \item $\mathcal{H}_0$ is dense in $(0,1)$.
 \item If there exists an integer $n\geq 1$ such that $\mathcal{H}_n $ is dense in $(1,\infty)$ then $\mathcal{H}_1 $ is also dense in $(1,\infty)$. \label{pnt:dense 2}
 \item If $\mathcal{H}_1$ is discrete (i.e without limit point) then $\mathcal{H}_n$ is also discrete for all integers $n\geq 1$.
\end{enumerate}
\end{theorem}

\begin{proof}{\ }
 \begin{enumerate}
 \item The first author and Smyth proved in \cite{HichriSmyth2024} that the set of non-Parry Perron numbers is dense in $(0,1)$. In the proof, they showed that any real number $\ell>1$ can be approached by an infinite sequence of Perron numbers $(\beta_n)$ where $\beta_n \in \overline{\mathbb{P}}_\Phi$, for all $n$ sufficiently large. Hence, we conclude by noting that $\overline{\mathbb{P}}_\Phi \subset \mathcal{H}_0$.
\item Let $\ell \geq 1 $ be a real number. As $\mathcal{H}_n$ is dense in $(1,\infty)$, there exists an infinite sequence $(\beta_k) \subset \mathcal{H}_n$ such that $\lim_{k\to\infty}(\beta_k)= \sqrt[n]{\ell} $. Hence, $\lim_{k\to\infty}(\beta_k)^n= \ell $. We conclude by noting that, using the second point of Theorem \ref{Parryorder} for $m=n$, we have $\beta_k^n\in \mathcal{H}_1$ for all integers $k \geq 0$.
 \item This proof is similar to part \ref{pnt:dense 2}.
 \end{enumerate}
\end{proof}
 
\begin{remark} {\ }
\begin{enumerate}
 \item We conjecture that $\mathcal{H}_1$ is dense.  This is based on a computational search for $\beta \in \mathcal{H}_1 \cap I_n$ with $\{I_n\}_n$ a sequence of small sub-intervals of $[1,2]$, combined with $\theta_d$ from Section \ref{sec:beta, theta} for various $d$.
 Based upon this search we found a large set $B = \{\beta_1 < \beta_n < \dots < \beta_{10887}\}$ with $\beta_i \in \mathcal{H}_1 \cap [1,2]$.
 We have $b_1 \approx 1.00490$, $\beta_{10997} \approx 1.99951$ and $\max \beta_{i+1}-\beta_i \approx 0.0026522$. 
 It is not clear how to prove that $\mathcal{H}_1$ is dense.
 \item It is useful to note that for every integer $n \geq 2$, $\mathcal{H}_n$ can be injected into $\mathcal{H}_1$ via the map: $\beta \to \beta^n$ according to part \ref{part:Parryorder 2} of Theorem \ref{Parryorder}. 
 \item For the other part, according to \eqref{P-decomposition} and \ref{H01}, we have $$\mathcal{P}= \mathcal{H}_1 \bigcup_{n \geq 2} (\mathcal{H}_n \cap \mathcal{P}) \bigcup (\mathcal{H}_\infty \cap \mathcal{P}).$$ Boyd conjectured in \cite{Boyd1977} that $\mathcal{H}_\infty =\mathcal{S}\cup \mathcal{T}$ is closed.  It is known that $\mathcal{P}$ is dense in $(1,\infty)$.  This means that we should have ${\bigcup_{n \geq 1} (\mathcal{H}_n \cap \mathcal{P}})$ dense in $(1,\infty)$. 
 \end{enumerate}
 \end{remark} 

\section{Computation of Parry order}
\label{sec:3}

\begin{example}
 \label{cor:Solomyak}
 Let $\beta$ be the Perron root of one of the polynomials in Table \ref{tab:perron}. 
 Then $\beta$ is not a Parry number, $\beta^2$ is a Parry number, and for all $k \geq 3$ we have $\beta^k$ is not a Parry number. 
 I.e $\SpecP(\beta) = \{2\}$ and $\beta \in \mathbb{P}_a \cap \mathcal{H}_2$.

 The techniques used for each of these examples is similar for all polynomials.
 We will provide the details only for the first polynomial, and sketch the key details for the remaining polynomials.

 Consider $\beta \approx 1.2528$, the root of $p(x) = z^{12}-z^{10}-z^9+z^4+z^3-z-1$.
 Let $\gamma \approx .4938 +.9056 i$ be the conjugate of $\beta$.

 One can check that taking $n = 82$ that $|\sigma(T^n_\beta(1))| > \frac{\lfloor \beta \rfloor}{|\gamma|-1}$.
 Hence $\beta$ is non-Parry by Lemma \ref{lem:Akiyama}.

 One can further check that 
 \[ d_{\beta^2}(1) = 1 0 1 (0 0 1 0 1 0 0 1 0 1 0 0)^{\omega}.\]
 Hence $\beta^2$ is a Parry number.

 We see that $K(\beta) = \left \lceil \frac{\log \theta_0}{\log |\gamma|}\right\rceil = 43$.
 By Lemma \ref{lem:K} we have $\beta^k$ is non-Parry for all $k \geq 43$. 
 For $k = 3 \dots 42$ we can explicitly check each case for find an $n$ such that $\left|\sigma(T^n_{\beta^k}(1))\right| > \frac{\lfloor \beta^k \rfloor}{|\gamma|^k -1}$.
 These are summarized in Table \ref{tab:K}.

 In Table \ref{tab:perron} we have indicated the Perron number $\beta$, its conjugate $\gamma$, and its minimal polynomial.
 We have provided the $\beta^2$ expansion of $1$, as well as the $K(\beta)$ as defined in Lemma \ref{lem:K} associated to proving $\beta$ is non-Parry. 
The details for the other four polynomials are similar.

 \begin{table}\begin{center}
 \begin{tabular}{|llp{4cm}lp{1.6cm}|} \hline
 $\beta$ & $\gamma$ & minimal poly & $d_{\beta^2}(1)$ & $K(\beta)$ from Lem. \ref{cor:Solomyak} \\
\hline
1.253 & .4938-.9056 i & $x^{12}-x^{10}-x^{9}+x^{4}+x^{3}-x-1$ & $1 0 1 (0 0 1 0 1 0 0 1 0 1 0 0)^{\omega}$ & 16\\ &&&& \\
1.353 & -.2561-1.008 i & $x^{12}-2x^{11}+2x^{10}-2x^{9}+x^{8}-x^{6}+x^{5}-x^{4}+x^{3}-x^{2}+x-1$ & $1 1 0 1 1 0 1 (0 0 1 0 0 0 0)^{\omega}$ & 13\\ &&&& \\
1.373 & .5457-.9013 i & $x^{11}-x^{8}-x^{7}-x^{6}-x^{5}-x^{4}+x^{2}+x+1$ & $1 1 1 0 0 1 1 0 1 1 0 0 (0 0 1)^{\omega}$ & 10\\ &&&&\\
1.408 & .7876-.6943 i & $x^{13}-2x^{12}+x^{11}-x^{7}+x^{6}-x^{5}+x^{4}-x^{3}+x-1$ & $1 1 1 1 1 0 1 1 1 0 (0 0 1 1)^{\omega}$ & 10\\ &&&&\\
1.414 & -.7811-.6589 i & $x^{14}-x^{11}-x^{10}-x^{9}-x^{8}-x^{7}-x^{6}+x^{3}+x^{2}+x+1$ & $1 1 1 1 1 1 1 1 1 1 1 0 1 0 0 0 1 (1 1 0 0)^{\omega}$ & 23\\ \hline
 \end{tabular}
 \caption{Example \ref{cor:Solomyak} where $\beta^k$ is Parry if and only if $k = 2$}\label{tab:perron}\end{center}
 \end{table}

\begin{table} 
 \begin{center} 
 \begin{tabular}{|ll|ll|ll|} 
 \hline 
 $k$ & $n$ & $k$ & $n$ &$k$ & $n$ \\ \hline 
1 & 82 & 7 & 7 & 13 & 3 \\ 
2 & NA & 8 & 6 & 14 & 3 \\ 
3 & 21 & 9 & 6 & 15 & 3 \\ 
4 & 11 & 10 & 4 & 16 & 2 \\ 
5 & 17 & 11 & 3 & &\\ 
6 & 6 & 12 & 2 & &\\ 
\hline
\end{tabular}
 \end{center} 
 \caption{Example \ref{cor:Solomyak} value of $n$ from Lemma \ref{lem:Akiyama} for $x^{12}-x^{10}-x^{9}+x^{4}+x^{3}-x-1$} \label{tab:K}
 \end{table} 
\end{example}

\begin{remark}
 Although, it may be well known that a Parry number could have infinitely many powers that are not Parry numbers and it is easy to provide several examples. As far as we know, no known example of non-Parry number whose powers include Parry numbers is already known. Here, Example \ref{cor:Solomyak} provides the first examples which we found to this problem.
\end{remark}

\subsection[Some comments on powers of beta and theta]{Some comments on $\beta_d^k$ and $\theta_d^k$} \label{sec:beta, theta}

The two polynomial families $x^d - x - 1$ and $x^d - x^{d-1} - 1$ are often considered in many works to show important results on the dynamic properties of Perron numbers \cite{Akiyama2016,HichriSmyth2024,VergerGaugry2008}. For instance, Akiyama proved that all Perron numbers $\beta_d$ of minimal polynomial $x^d - x - 1 = 0$, with $d \ge 4$ are non-Parry numbers belonging to $\mathbb{P}_\Phi$. In contrast, it is known, that the roots $\theta_d$ of $x^d - x^{d-1} - 1 = 0$ are always Parry numbers for $d\geq 2$ with $d_{\theta_d}(1)=1\underbrace{0\dots 0}_{d-2}1$. Moreover, the second author and Smyth in \cite{HichriSmyth2024}, used powers of $\beta_d$ from the first family to construct explicit non-Parry Perron numbers, proving that non-Parry Perron numbers are dense in $(1, \infty)$ and that their conjugates are dense in $\mathbb{C}$. Here we add some computational results to support the following: 
\begin{conj} {\ }
\label{conj:1}
For all $d \geq 4$ and all $k \geq 1$ we have $\beta_d^k$ is non-Parry.
\end{conj}

In Table \ref{tab:beta} we consider $\beta_d^k$, giving numerical support for the conjecture

\begin{table}
\begin{center}
\begin{tabular}{|lllp{110pt}l|} \hline
 $\beta$ &$\gamma$&   Minimal  & $d_{\beta^k}(1)$ & $K(\beta)$ from Cor. \ref{cor:Solomyak} \\
 & & polynomial & & \\
 \hline
 1.618 & NA & $x^{2}-x-1$ & Is Pisot & NA \\
1.325 & NA & $x^{3}-x-1$ & Is Pisot & NA \\ 
1.221 & -.2481-1.034i & $x^{4}-x-1$ & Non-Parry for $k \geq 1$ &8\\
1.167 & .1812-1.084i & $x^{5}-x-1$ & Non-Parry for $k \geq 1$ &6\\
1.135 & .4511-1.002i & $x^{6}-x-1$ & Non-Parry for $k \geq 1$ &6\\
1.113 & .6171-.9009i & $x^{7}-x-1$ & Non-Parry for $k \geq 1$ &6\\
1.097 & .7232-.8071i & $x^{8}-x-1$ & Non-Parry for $k \geq 1$ &6\\
1.085 & .7941-.7263i & $x^{9}-x-1$ & Non-Parry for $k \geq 1$ &7\\
1.076 & .8431-.6579i & $x^{10}-x-1$ & Non-Parry for $k \geq 1$ &8\\ \hline
\end{tabular}
\caption{Root $\beta_d$ of $x^d-x-1$.}
\label{tab:beta}
\end{center}
\end{table}

\begin{remark}
 It is worth remarking that this has been verified up to $d = 100$. 
\end{remark}

Similarly in Table \ref{tab:theta} we consider $\theta_d^k$.
We note that $x^5-x^4-1$ is reducible, with non-trivial factor $x^3-x-1$.
This is in fact true for all $d \equiv 5 \pmod 6$, as 
 $x^2-x+1 | x^{6k+5} - x^{6k+4} - 1$.
For all tests we use the non-cyclotomic factor of $x^{d}-x^{d-1}-1$.

We make the conjecture:
\begin{conj}
 \label{conj:3}
 For all $d \geq 6$ we have $\theta_d$ is Parry and for all $k \geq 2$ we have $\theta_d^k$ is non-Parry.
\end{conj}

\begin{remark}
 As before, we have verified this up to $d = 100$.
\end{remark}

\begin{table}
\begin{center}
\begin{tabular}{|lllp{110pt}l|} \hline
 $\beta$ & $\gamma$ & Minimal & $d_{\beta}(1)$ & $K(\beta)$ from Cor. \ref{cor:Solomyak} \\
 & & polynomial & & \\
 \hline
1.618 & NA & $x^{2}-x-1$ & Is Pisot & NA \\
1.466 & NA & $x^{3}-x^{2}-1$ & Is Pisot & NA \\
1.380 & NA& $x^{4}-x^{3}-1$ & Is Pisot & NA \\
1.325 & NA & $x^{3}-x-1$ & Is Pisot & NA \\
1.285 & .6714-.7849i & $x^{6}-x^{5}-1$ & $1 0 0 0 0 1 (0)^{\omega}$ \newline Non-Parry for $k \geq 2$ &15\\
1.255 & .7802-.7053i & $x^{7}-x^{6}-1$ & $1 0 0 0 0 0 1 (0)^{\omega}$ \newline Non-Parry for $k \geq 2$ &10\\
1.232 & .8522-.6353i & $x^{8}-x^{7}-1$ & $1 0 0 0 0 0 0 1 (0)^{\omega}$ \newline Non-Parry for $k \geq 2$ &8\\
1.213 & .9017-.5753i & $x^{9}-x^{8}-1$ & $1 0 0 0 0 0 0 0 1 (0)^{\omega}$ \newline Non-Parry for $k \geq 2$ &8\\
1.197 & .9368-.5243i & $x^{10}-x^{9}-1$ & $1 0 0 0 0 0 0 0 0 1 (0)^{\omega}$ \newline Non-Parry for $k \geq 2$ &7\\ \hline
\end{tabular}
\caption{Root $\theta_d$ of $x^d-x^{d-1}-1$.}
\label{tab:theta}
\end{center}
\end{table}


It is already known that a degree 2 Perron number is a Parry number if and only if it is Pisot  number. Consequently, if  $\beta$ is a degree 2 Perron number then  we have  one of  the two following cases:
\begin{enumerate}
    \item  $\OrdP(\beta) = \infty$ if $\beta$  is a Pisot number.
    \item  $\OrdP(\beta) = 0$ otherwise.
\end{enumerate}

However, as shows the following section, the case of degree 3 Perron number is much more complicate. 

\subsection{Cubic Perron numbers}

It is possible to give a partial description of cubic Perron numbers and their Parry order.  

\begin{proposition}
    Let $\beta$ be a degree 3 Perron number with two complex roots. 
    Then:
    \begin{enumerate}
        \item If $|\gamma| < 1$ then $\beta$ is Pisot and $\OrdP(\beta) = \infty$.
        \item If $|\gamma| > 1$ then $\OrdP(\beta) = 0$.
    \end{enumerate}
\end{proposition}

\begin{proof}
    The first part is obvious.

    Assume that $|\gamma| > 1$.
    Then $|\mathrm{Norm}(\beta)| > \beta$ and $\beta$ is 
    not Parry by Boyd's critieria.
    Further, $|\mathrm{Norm}(\beta^n)| > \beta^n$ for all $n$ and hence $\OrdP(\beta) = 0$.   
\end{proof}

\begin{proposition} \label{prop:cubic}
    Let $\beta$ be a degree $3$ totally real Perron number with conjugates $\alpha$ and $\gamma$. Assume $|\alpha| \leq |\gamma|$.  Then
    \begin{enumerate}
        \item If $|\alpha|, |\gamma| < 1$ then $\beta$ is a Pisot number and $\OrdP(\beta) = \infty$. \label{prt:cubic 1}
        \item If $-|\beta| < \gamma < -1 < \alpha < 1$ then $\beta$ {\em might} have non-trivial Parry Order.\label{prt:cubic 2}
        \item If $-1 < \alpha < 1 < \gamma < \beta$ then $\OrdP(\beta) = 0$.\label{prt:cubic 3}
        \item If $1 < |\alpha|, |\gamma| < \beta$ then $\OrdP(\beta) = 0$.\label{prt:cubic 4}
    \end{enumerate}
\end{proposition}

\begin{proof}
Part \ref{prt:cubic 1} is obvious.

To see Part \ref{prt:cubic 2}, we notice that $\beta_1$, the root of $x^3 - 5 x^2 - 2 x + 5$ and $\beta_2$, the root of $x^3 - 5 x^2 - 8 x + 5$ both satisfy $|\beta_i| < \gamma_i < -1 < \alpha_i < 1 < \beta_i$ for $i = 1, 2$.
We note that $d_{\beta_1}(1) = 5(1\ 0)^\omega$, and $\SpecP(\beta_1) = \{1, 3, 5\}$.
We note that $|\gamma_2| > \frac{1+\sqrt{5}}{2}$, and hence by Lemma \ref{lem:K} we gave $\OrdP(\beta_2) = 0$.

To see Part \ref{prt:cubic 3} we assume that $\beta^k$ has a Parry expansion.  
Let $d_{\beta^k}(1) = \sum a_i \beta^{-ki}$ be this finite or periodic expansion.
As the expansion is finite or periodic, we have $d_{\beta^k}(1) \in \mathbb{Q}(\beta)$.
By considering the Galois action $\sigma(\beta) = \gamma$ we see that 
$1 = \sum a_i \gamma^{-ki}$.
As $1 < \gamma < \beta$ we $1 = \sum a_i \beta^{-ki} < \sum a_i \gamma^{-ki} = 1$, which is a contradiction.

Part \ref{prt:cubic 4} follows from Boyd's criteria.
\end{proof}

\begin{corollary}
    If $\beta$ is totally real non-Pisot cubic Perron number, then either $\OrdP(\beta) = 0$ or $\OrdP(\beta)$ is odd.
\end{corollary}

\begin{proof}
    By Proposition \ref{prop:cubic} we note that for $\beta$ to have non-trivial order we must have $-|\beta| < \gamma < -1 < \alpha < 1 < \beta$.
    We notice that $-|\beta^{2k+1}| < \gamma^{2k+1} < -1 < \alpha^{2k+1} < 1 < \beta^{2k+1}$ and $ -1 < \alpha^{2k} < 1 <  \gamma^{2k} < \beta^{2k}$.
    Hence if $\beta^m$ is Parry then $m$ is odd.
\end{proof}

\begin{theorem}[Parry, \cite{Par60}] \label{thm:parry}
    Let $\mathbf{a}=(a_i)_{i\geq 1}$ be a sequence in $\{0,1,\dots, M\}^{\mathbb{N}}$. That is, $\mathbf{a}$ is a non-empty finite or infinite sequence $a_1\dots a_k$ or $a_1a_2\dots$ with $a_i\in\{0,1,\dots,M\}$. Then the sequence $\mathbf{a}$ is the greedy expansion of $1$ for some $\beta>1$ if and only if for all $j\geq 1$
    $$\mathrm{shift}^j(\mathbf{a})<_{\text{lex}}\mathbf{a}.$$
    Here $<_{\text{lex}}$ is the lexigraphical order on strings, and  
    $\mathrm{shift}(a_1a_2\dots)=a_2a_3\dots$. 
\end{theorem}

\begin{proposition}
    Let $\beta$ be the Perron root of $x^3 - a x^2 - b x + c$ with $0 \leq c \leq a$, $1 \leq b \leq a -1$.
    Then $\beta$ is Parry.
\end{proposition}

\begin{proof}
    We claim that $d_{\beta}(1) = a ((b-1) (a-c))^\omega$.  
    Using the fact that $\beta^3 = a \beta^2 +b \beta - c$, we can show that 
    \[1 = \frac{a}{\beta} + \sum_{i=1}^\infty \left(\frac{b-1}{\beta^{2i}}+\frac{a-c}{\beta^{2i+1}}\right) \]
    This proves that $a ((b-1) (a-c))^\omega$ is a value $\beta$-representation of $1$. 

    The fact that this $\beta$-representation is greedy follows from Theorem \ref{thm:parry}. 
\end{proof}

\begin{remark}
In \cite[Theorem 2]{Bassino02} Bassino provides the $\beta$-expansion of cubic Pisot numbers.
The expansion given in the second part of Case 2 of this Theorem did not need the hypothesis that the cubic algebraic integer was a Pisot number.  It is equivalent to that given above.
\end{remark}

\begin{proposition}
    For all $N$ there exists a non-Pisot Parry number $\beta$ such that $\OrdP(\beta) > N$.
\end{proposition}
\begin{proof}
    Let $\gamma(d), \alpha(d)$ and $\beta(d)$ be the root of $x^3-d x^2-2 x + d$ where
    $\gamma(d) < -1 < \alpha(d) < 1 < \beta(d)$.

    We can show that
    \begin{align}
    \gamma(d) & = -1-\frac{1}{2d} + \frac{3}{8 d^2} - \frac{55}{128 d^4} + \dots, \label{eq:1} \\
    \alpha(d) & = 1-\frac{1}{2d} - \frac{3}{8 d^2} + \frac{55}{128 d^4} + \dots, \label{eq:2} \\
    \beta(d) & = d + \frac{1}{d} - \frac{1}{d^5} + \frac{2}{d^7} - \frac{1}{d^9} + \dots.  \label{eq:3}
    \end{align}

    We observe from this that for sufficiently large $d$ that $\beta_d$ is not a Pisot number.  Hence it's Parry order is bounded for any fixed $d$.

    Letting $p_n(x) = (x-\beta(d)^n) (x-\alpha(d)^n) ( x- \gamma(d)^n) = x^3 - a_n(d) x^2 - b_n(d) x + c_n(d)$.

    Using equations \eqref{eq:1}, \eqref{eq:2} and \eqref{eq:3} we get 
    \begin{align*}
        a_n(d) & = 
            \gamma(d)^n + \beta(d)^n + \alpha(d)^n, \\
        b_n(d) & = -(\gamma(d)^n \beta(d)^n + \gamma(d)^n \alpha(d)^n + \alpha(d)^n \beta(d)^n). \\
    \end{align*}

    Let $d$ be odd, and consider $a_d(d^2)$ and $b_d(d^2)$. 
    This simplify to 
    \begin{align*}
        a_d(d^2) & = d^{2d} \left( 1 + \frac{1}{d^2} + \frac{1}{2 d^6} - \frac{1}{2 d^7} + \frac{1}{6 d^9} + \dots \right), \\
        b_d(d^2) & = d^{2d} \left(-\frac{1}{d} - \frac{1}{24 d^3} - \frac{1}{2 d^4} - \frac{881}{1920 d^5} - \frac{1}{48 d^6} - \dots \right).         
    \end{align*}
    In particular this implies that for large odd $d$ that we have $\OrdP(\beta_{d^2}) \geq d$.

This proves the desired result.
\end{proof}

\begin{example}
    Let $P_d(x) = x^3 - d x^2 - 2 x + d$ have Parry root $\beta_d > d$.
Below we have listed the Parry Spectrum for $1 \leq d \leq 100$. 
\[ \SpecP(\beta_d) = \begin{cases}
    \{1, 3, \dots, d\} & d \text{ odd and } 1 \leq d \leq 9 \\
    \{1, 3, \dots, d-2\} & d \text{ odd and } 11 \leq d \leq 63 \\
    \{1, 3, \dots, d-4\} & d \text{ odd and } 65 \leq d \leq 99 \\
    \{1, 3, \dots, d-1\} & d \text{ even and } 2 \leq d \leq 36 \\
    \{1, 3, \dots, d-3\} & d \text{ even and } 38 \leq d \leq 90 \\
    \{1, 3, \dots, d-5\} & d \text{ even and } 92 \leq d \leq 100
\end{cases}.\]
We note that the Parry Order is higher than what is given in Proposition \ref{prop:cubic}.
\end{example}


\subsection{Some further data}
\label{sec:ParryVsNonParry}

Let $Q_N^d$ be the set of Perron numbers that divide a degree $d$ polynomials $x^{d}+ a_{d-1} x^{d-1} + \dots a_0$ with $|a_i| \leq N$ with the additional property that the largest
    conjugate is bounded by $\frac{1+\sqrt{5}}{2}$.
We have listed in Table \ref{tab:Q} the size of $Q^{d}_N$ for various $d$ and $N$, 
    along with how many of these polynomials are Pisot, Parry, Salem, and the various 
    orders found.

\begin{table}
\begin{tabular}{lllllll}
$d$ & $N$ & $\# Q_{d}^N$ & $\# Q_{d}^N \cap \mathcal{P} $& $\# Q_{d}^N \cap \mathcal{S} $& $\# Q_{d}^N \cap \mathcal{T} $& $\# Q_{d}^N \cap \mathcal{H}_n $\\ \hline 
2& 50 &  3374&  2451&  2451&  0 & $\# \cdot \cap \mathcal{H}_{0} = 923$, $\# \cdot \cap \mathcal{H}_{\infty} = 2451$ \\ \hline 
3& 10 &  2499&  1395&  1185&  0 & $\# \cdot \cap \mathcal{H}_{0} = 1104$, $\# \cdot \cap \mathcal{H}_{1} = 181$, \\ 
 &&&&&& $\# \cdot \cap \mathcal{H}_{3} = 20$, $\# \cdot \cap \mathcal{H}_{5} = 5$, \\ 
 &&&&&& $\# \cdot \cap \mathcal{H}_{7} = 2$, $\# \cdot \cap \mathcal{H}_{9} = 2$, \\ 
 &&&&&& $\# \cdot \cap \mathcal{H}_{\infty} = 1185$ \\ \hline 
4& 5 &  3288&  1230&  745&  33 & $\# \cdot \cap \mathcal{H}_{0} = 2042$, $\# \cdot \cap \mathcal{H}_{1} = 329$, \\ 
 &&&&&& $\# \cdot \cap \mathcal{H}_{2} = 65$, $\# \cdot \cap \mathcal{H}_{3} = 21$, \\ 
 &&&&&& $\# \cdot \cap \mathcal{H}_{4} = 15$, $\# \cdot \cap \mathcal{H}_{5} = 12$, \\ 
 &&&&&& $\# \cdot \cap \mathcal{H}_{6} = 9$, $\# \cdot \cap \mathcal{H}_{7} = 2$, \\ 
 &&&&&& $\# \cdot \cap \mathcal{H}_{8} = 7$, $\# \cdot \cap \mathcal{H}_{9} = 1$, \\ 
 &&&&&& $\# \cdot \cap \mathcal{H}_{10} = 2$, $\# \cdot \cap \mathcal{H}_{12} = 1$, \\ 
 &&&&&& $\# \cdot \cap \mathcal{H}_{14} = 2$, $\# \cdot \cap \mathcal{H}_{16} = 1$, \\ 
 &&&&&& $\# \cdot \cap \mathcal{H}_{17} = 1$, $\# \cdot \cap \mathcal{H}_{\infty} = 778$ \\ \hline 
5& 3 &  3222&  869&  431&  0 & $\# \cdot \cap \mathcal{H}_{0} = 2327$, $\# \cdot \cap \mathcal{H}_{1} = 308$, \\ 
 &&&&&& $\# \cdot \cap \mathcal{H}_{2} = 70$, $\# \cdot \cap \mathcal{H}_{3} = 37$, \\ 
 &&&&&& $\# \cdot \cap \mathcal{H}_{4} = 13$, $\# \cdot \cap \mathcal{H}_{5} = 7$, \\ 
 &&&&&& $\# \cdot \cap \mathcal{H}_{6} = 7$, $\# \cdot \cap \mathcal{H}_{7} = 5$, \\ 
 &&&&&& $\# \cdot \cap \mathcal{H}_{8} = 4$, $\# \cdot \cap \mathcal{H}_{9} = 3$, \\ 
 &&&&&& $\# \cdot \cap \mathcal{H}_{10} = 2$, $\# \cdot \cap \mathcal{H}_{11} = 1$, \\ 
 &&&&&& $\# \cdot \cap \mathcal{H}_{12} = 2$, $\# \cdot \cap \mathcal{H}_{14} = 1$, \\ 
 &&&&&& $\# \cdot \cap \mathcal{H}_{15} = 1$, $\# \cdot \cap \mathcal{H}_{18} = 1$, \\ 
 &&&&&& $\# \cdot \cap \mathcal{H}_{21} = 1$, $\# \cdot \cap \mathcal{H}_{22} = 1$, \\ 
 &&&&&& $\# \cdot \cap \mathcal{H}_{\infty} = 431$ \\ \hline 
6& 2 &  2306&  472&  160&  18 & $\# \cdot \cap \mathcal{H}_{0} = 1823$, $\# \cdot \cap \mathcal{H}_{1} = 223$, \\ 
 &&&&&& $\# \cdot \cap \mathcal{H}_{2} = 39$, $\# \cdot \cap \mathcal{H}_{3} = 10$, \\ 
 &&&&&& $\# \cdot \cap \mathcal{H}_{4} = 9$, $\# \cdot \cap \mathcal{H}_{5} = 3$, \\ 
 &&&&&& $\# \cdot \cap \mathcal{H}_{6} = 8$, $\# \cdot \cap \mathcal{H}_{7} = 4$, \\ 
 &&&&&& $\# \cdot \cap \mathcal{H}_{8} = 4$, $\# \cdot \cap \mathcal{H}_{10} = 3$, \\ 
 &&&&&& $\# \cdot \cap \mathcal{H}_{12} = 2$, $\# \cdot \cap \mathcal{H}_{\infty} = 178$ \\ \hline 
\end{tabular}
\caption{Data for families of Perron numbers} \label{tab:Q}
\end{table}

 \section{ Conclusion }
\label{sec:4}

 Basing on some old known results and all the theoretical and experimental results of this paper, the dynamical behavior of powers of Perron numbers is as follows: Let $\beta $ be a Perron number then we have one of the following cases: 
\begin{itemize}
 \item If $\beta $ is a Pisot number or degree 4 Salem number then $\beta^n $ is a Parry number for all integers $n$.
 
 \item Basing on Boyd's model prediction, we expect that if $\beta $ is a Salem number of degree $\geq 8$, then $\beta^n$ will be Parry and non-Parry numbers for infinitely many integers $n$. 
 \item If $\beta $ belongs to $\mathbb{P}_\Phi$ but it is neither a Pisot nor Salem number then $\beta^n $ can alternate for finitely many cases (for all integers $ 1\leq n \leq \OrdP(\beta) $) between the sets Parry and non-Parry numbers. It will be a non-Parry number for all integers $n >\OrdP(\beta)$.

 \item If $\beta \in \overline{\mathbb{P}}_\phi$, then $\beta^n$ is non-Parry number for all integers $n\geq 1$. 

\end{itemize}
 
Denote by $D_n=\{z \in \mathbb{C}, ~~ \mathrm{ such ~~that} ~~|z|< \sqrt[n]{\theta_0}\}$ and by $D_\infty=\mathbb{D}$ the unit disk. 
 Then, according to Solomyak criteria and Lemma \ref{lem:perronpowers}, it is clear that $\OrdP(\beta)=n$ implies the the set of Galois conjugates of $\beta $, other than $\beta$, is contained in the disc $D_n=\{z \in \mathbb{C}, ~~ \mathrm{ such ~~that} ~~|z|< \sqrt[n]{\theta_0}\}$. In particular, when $n \to \infty$, we have $\lim_{n \to \infty}(D_n) = D_\infty = \mathbb{D}$ and then the Galois conjugates of $\beta$ approach the unit disc. 

It is interesting to note in Table \ref{tab:Q} that the 
     proportions of $\mathcal{H}_0, $ increase but those     
    of $\mathcal{H}_1, $ and $\mathcal{H}_\infty, $ decrease when $d$ increase.
It would be interesting to know if this is a general phenomenon or not.

 Finally, according to Section \ref{sec:1}, the research concerning the small Salem number has been been a major focus of various research programs. This question is equivalent to the research concerning small Perron number $\beta$ satisfying $\OrdP(\beta)=\infty$. 
 In all cases, we hope that this result provides a new approach to the out longstanding problem of research of smallest Salem number.

\end{document}